\documentclass[11pt]{amsart}
\usepackage{latexsym,amsfonts,amssymb,amsthm,amsmath,mathrsfs,color,amscd,graphicx,fullpage,fancyhdr,parskip,hyperref}

\newcommand{\bb}[1]{{\mathbb #1}}

\newtheorem{theorem}{Theorem} 
\newtheorem{corollary}[theorem]{Corollary}

\newtheorem{lemma}[theorem]{Lemma}
\newtheorem{proposition}[theorem]{Proposition}

\newtheorem{remark}[theorem]{Remark}

\numberwithin{equation}{section}
\numberwithin{theorem}{section}

\begin{document}

\title{Existence and uniqueness of solutions for Bertrand and Cournot mean field games}
\author{P. Jameson Graber}
\thanks{Both authors are grateful to be supported in this work by the National Science Foundation under NSF Grant DMS-1303775.}
\address{International Center for Decision and Risk Analysis\\
Naveen Jindal School of Management\\
The University of Texas at Dallas\\
800 West Campbell Rd, SM30\\
Richardson, TX 75080-3021\\
Phone: (972) 883-6249}
\email{pjg140130@utdallas.edu}

\author{Alain Bensoussan}
\thanks{This research is also supported by the Research Grants Council of HKSAR (CityU 500113).}
\address{International Center for Decision and Risk Analysis\\
Naveen Jindal School of Management\\
The University of Texas at Dallas\\
800 West Campbell Rd, SM30\\
Richardson, TX 75080-3021\\
Phone: (972) 883-6117}
\address{Department of Systems Engineering and Engineering Management\\
 City University of Hong Kong\\
 83 Tat Chee Avenue\\
 Kowloon Tong, Hong Kong}
\email{axb046100@utdallas.edu}

\dedicatory{Version: \today}

\maketitle

\begin{abstract}
We study a system of partial differential equations used to describe Bertrand and Cournot competition among a continuum of producers of an exhaustible resource.
By deriving new a priori estimates, we prove the existence of classical solutions under general assumptions on the data.
Moreover, under an additional hypothesis we prove uniqueness.


\end{abstract}

\section{Introduction} \label{sec:intro}

Our purpose is to study the following system of partial differential equations:
\begin{equation}
\label{main system}
\left\{
\begin{array}{rcc}
(i) & u_t + \frac{1}{2}\sigma^2 u_{xx} - ru + H(t,u_x,[mu_x]) = 0, & 0 < t < T, \ 0 < x < L\\
(ii) & m_t - \frac{1}{2}\sigma^2 m_{xx} - \left(G(t,u_x,[mu_x])m\right)_x = 0, & 0 < t < T, \ 0 < x < L\\
(iii) & m(0,x) = m_0(x), u(T,x) = u_T(x), & 0 \leq x \leq L\\
(iv) & u(t,0) = m(t,0) = 0, ~~ u_x(t,L) = 0, & 0 \leq t \leq T \\ 
(v) & \frac{1}{2}\sigma^2 m_x(t,L) + G(t,u_x(t,L),[mu_x])m(t,L) = 0, & 0 \leq t \leq T
\end{array}\right.
\end{equation}
where $T > 0$ and $L > 0$ are given constants, $m_0$ and $u_T$ are known smooth functions, and $H$ and $G$ are defined below in Section \ref{sec:specification}.
We mention for now that $H$ and $G$ depend on the variable $mu_x$ in a {\em nonlocal} way, in particular they are functions of
$$
\int_0^L u_x(t,x)m(t,x)dx.
$$
System \eqref{main system} was introduced by Chan and Sircar in \cite{chan2014bertrand} to represent a mean field game in which producers compete to sell an exhaustible resource.
Here we view the producers as a continuum of rational agents whose ``density" is given by the function $m(t,x)$ governed by a Fokker-Planck equation.
Each of them must solve an optimal control problem corresponding to the Hamilton-Jacobi-Bellman equation \eqref{main system}(i).
Further details will be given below in Section \ref{sec:specification}.

Mean field games were introduced in \cite{lasry07,huang2006large} to describe differential games with large numbers of players represented by a continuum.
Most recent results deal with models of the form
\begin{equation*}
\begin{array}{c}
u_t + \frac{1}{2}\sigma^2 u_{xx} - ru + H(t,x,u_x) = V[m]\\
m_t - \frac{1}{2}\sigma^2 m_{xx} - \left(G(t,x,u_x)m\right)_x = 0
\end{array}
\end{equation*}
where $V[m]$ is a monotone function.
There have been a number of existence and uniqueness theorems proved when $V[m]$ depends on $m$ both locally \cite{porretta2013weak,cardaliaguet2014second,cardaliaguet2013weak,cardaliaguet2014mean,cardaliaguet2012long,gomes2013time,gomes2014time} and non-locally \cite{cardaliaguet2013long2}.
More recently some progress has been made toward similar results for cases where the Hamiltonian depends on $m$ in a nonlinear way \cite{gomes2014existence,gomes2015short,graber2015weak}.
However, none of these results address systems where the coupling happens in the nonlocal part of the Fokker-Planck equation.

Applications of mean field games to economics have attracted much recent interest; see \cite{achdou2014partial,burger2014partial,gomes2014socio} for surveys of the topic.
The model we study here, which comes from \cite{chan2014bertrand}, describes Bertrand or Cournot economic competition in the mean field limit (i.e.~for a continuum of producers/consumers).
It resembles the model proposed by Gu\'eant, Lasry, and Lions to model oil production \cite{gueant11}, and it appears in a more complex and highly nonlinear form in \cite{chan2015fracking} to describe the response of traditional oil producers to new technological developments such as renewable energy and ``fracking".
In \cite{chan2015fracking} the authors point out, ``There does not exist anything like general existence and
uniqueness theorems for PDE systems of this kind."
This statement has inspired the present work, in which we prove existence and uniqueness for \eqref{main system}.

We would like to mention in this context a recent result of Burger, et al.~ \cite{burger2015boltzmann} which provides an existence and uniqueness theorem for a mean field games model of knowledge growth introduced by Lucas and Moll \cite{lucas11}.
Their model also involves a coupled Boltzmann/Hamilton-Jacobi system of equations in which the coupling occurs through an integral over the space variable.
This structure is natural for applications to economics, since aggregate quantities such as market price or total production are expressed mathematically as averages with respect to the density of agents.
For this reason it is desirable to develop techniques to analyze PDE systems of this type.

In this article we prove that under general conditions there exists a classical solution to \eqref{main system}, which under a certain restriction is also unique.
By ``classical solution" we mean that the equations in \eqref{main system} hold pointwise.
We consider only the case where $\sigma > 0$ so that the equations are of parabolic type.
Existence is obtained by applying the Leray-Schauder fixed point theorem; accordingly, the main effort of this paper is to provide {\em a priori estimates} of solutions.
A new feature of our analysis is the estimation of the nonlocal term $\int_0^L u_x(t,x)m(t,x)$.
Although traditional methods provide estimates of $u$ in $L^\infty$, it is not immediately clear how to obtain similar estimates for the gradient $u_x$.
In Section \ref{sec:nonlocal term} we exploit the structure of \eqref{main system} by directly computing the time derivative of the nonlocal term, careful analysis of which allows us to derive higher order regularity.

The remainder of this paper is organized as follows.
In the rest of the introduction we give definitions of the functions $H$ and $G$ from \eqref{main system}, introduce some notation and give our main assumptions on the data.
Section \ref{sec: a priori estimate} is devoted to a priori estimates and constitutes the core of this paper.
In Section \ref{sec:existence} we prove the existence of solutions.
Finally, in Section \ref{sec:uniqueness} we prove uniqueness under an additional hypothesis.

\subsection{Specification and explanation of the model} \label{sec:specification}

We summarize the interpretation of \eqref{main system} as follows.
Let $t$ be time and $x$ be the producer's capacity.
We assume there is a large set of producers and represent it as a continuum.
We say $m(t,x)$ is the ``density" of producers at time $t$, so that
\begin{equation}
\label{average}
\eta(t) := \int_0^L m(t,x)dx, ~~ 0 \leq \eta(t) \leq 1
\end{equation}
represents the total mass of producers remaining with  positive stock.
Note that $\eta(t)$ is a decreasing function in time.

The first equation in \eqref{main system} is the Hamilton-Jacobi-Bellman (HJB) equation for the maximization of profit.
Each producer's capacity is driven by a stochastic differential equation
\begin{equation}
dX(s) = -q(s,X(s))ds + \sigma {\bf 1}_{X(s) > 0} dW(s),
\end{equation}
where $q$ is determined by the price $p$ through a linear demand schedule
\begin{equation}
\label{demand function}
q = D^{(\eta)}(p,\bar{p}) = a(\eta) - p + c(\eta)\bar{p}, ~~ \eta > 0.
\end{equation}
In \eqref{demand function} $\bar{p}$ represents the market price, that is, the average price offered by all producers.
This is given by
\begin{equation} \label{def average price}
\bar{p}(t) = \frac{1}{\eta(t)}\int_0^L p^*(t,x)m(t,x)dx,
\end{equation}
where $p^*(t,x)$ is the Nash equilibrium price.
The coefficients $a(\eta)$ and $c(\eta) = 1 - a(\eta)$ are defined by
\begin{equation}
a(\eta) := \frac{1}{1+\epsilon \eta}, ~~ c(\eta) := \frac{\epsilon \eta}{1+\epsilon \eta}
\end{equation} 
for a given fixed competition parameter $\epsilon \geq 0$.
The case $\epsilon = 0$ corresponds to monopoly, while perfect competition is given by $\epsilon = +\infty$.
Thus each producer competes with all the others by responding to the market price.

We define the value function
\begin{equation} \label{optimal control problem}
u(t,x) := \sup_{p} \bb{E} \left\{
\int_t^T e^{-r(s-t)}p(s)q(s) ds + u_T(X(T)) \ | \  X(t) = x \right\}
\end{equation}
where $q(s)$ is given in terms of $p(s)$ by \eqref{demand function}.
The optimization problem \eqref{optimal control problem} has the corresponding Hamilton-Jacobi-Bellman equation
\begin{equation}
\label{bellman}
u_t + \frac{1}{2}\sigma^2 u_{xx} - ru + \max_{p } \left[\left(a(\eta(t)) - p + c(\eta(t))\bar{p}(t)\right)(p-u_x)\right] = 0.
\end{equation}
The optimal $p^*(t,x)$ satisfies the first order condition
\begin{equation} \label{def equilibrium price}
p^*(t,x) = \frac{1}{2}\left(a(\eta(t)) + c(\eta(t))\bar{p}(t) + u_x(t,x) \right),
\end{equation}
and we take $q^*(t,x)$ to be the corresponding demand
\begin{equation}
q^*(t,x) = \frac{1}{2}\left(a(\eta(t)) + c(\eta(t))\bar{p}(t) - u_x(t,x) \right).
\end{equation}
This leads to Equation \eqref{main system}(i) where
\begin{equation}
\label{hamilton-jacobi}
 H(t,u_x,[mu_x]) := q^*(t,x)^2 = \frac{1}{4}\left(a(\eta(t)) + c(\eta(t))\bar{p}(t) - u_x\right)^2
\end{equation}
On the other hand, the density of producers is driven by the Fokker-Planck equation \eqref{main system}(ii), where
\begin{equation}
\label{boltzmann}
G(t,u_x,[mu_x]) := q^*(t,x) = \frac{1}{2}\left(a(\eta(t)) + c(\eta(t))\bar{p}(t) - u_x\right)
\end{equation}
The coupling takes place through the average price function, which, thanks to \eqref{def average price} and \eqref{def equilibrium price}, is given by
\begin{equation}
\label{average price}
\bar{p}(t) = \frac{1}{2-c(\eta(t))}\left(a(\eta(t)) + \frac{1}{\eta(t)}\int_0^L u_x(t,x)m(t,x)dx\right)
\end{equation}
We have taken Dirichlet boundary conditions at $x = 0$ as in \cite{chan2014bertrand}.
On the other hand, rather than taking $L = +\infty$ and working on an unbounded domain, we have taken Neumann boundary conditions at $x = L$, which represents a diffusion which is reflected at this boundary point.
We can think of $L$ as an upper limit on the capacity of any given producer.

\subsection{Notation and assumptions}

Throughout this article we define $Q_T := (0,T) \times (0,L)$ to be the domain, $S_T := ([0,T] \times \{0,L\}) \cup (\{T\} \times [0,L])$ to be the parabolic boundary, and at times $\Gamma_T := ([0,T] \times \{0\}) \cup (\{T\} \times [0,L])$ to be the parabolic half-boundary.
For any domain $X$ in $\bb{R}$ or $\bb{R}^2$ we define $L^p(X)$, $p \in [1,+\infty]$ to be the Lebesgue space of $p$-integrable functions on $X$; $C^0(X)$ to be the space of all continuous functions on $X$; $C^{\alpha}(X)$, $0 < \alpha < 1$ to be the space of all H\"older continuous functions with exponent $\alpha$ on $X$; and $C^{n+\alpha}(X)$ to be the set of all functions whose $n$ derivatives are all in $C^\alpha(X)$.
For a subset $X \subset \overline{Q_T}$ we also define $C^{1,2}(X)$ to be the set of all functions on $X$ which are locally continuously differentiable in $t$ and twice locally continuously differentiable in $x$.
By $C^{\alpha/2,\alpha}(X)$ we denote the set of all functions which are locally H\"older continuous in time with exponent $\alpha/2$ and in space with exponent $\alpha$.

We will denote by $C$ a {\em generic} constant, which depends only on the data (namely $u_T,m_0,L,T,\sigma,r$ and $\epsilon$).
Its precise value may change from line to line.

Throughout we take the following assumptions on the data:
\begin{enumerate}
\item $u_T(x)$ and $m_0(x)$ are functions in $C^{2+\gamma}([0,L])$ for some $\gamma > 0$.
\item $u_T$ and $m_0$ satisfy compatible boundary conditions: $u_T(0) = u_T'(L) = 0$ and $m_0(0) = m_0(L) = m_0'(L) = 0$.
\item $m_0 \geq 0$ and $\int_0^L m_0(x)dx = 1$, i.e.~ $m_0$ is a probability density.
\item $u_T \geq 0$ and $u_T' \geq 0$, i.e.~$u_T$ is non-negative and non-decreasing.
\end{enumerate}
\begin{remark}
Of all the assumptions, the stipulation that $u_T$ be non-negative and non-decreasing seems the least essential; it is not necessary for most estimates.
However, it appears to be needed to prove the a priori bounds of Section \ref{sec:nonlocal term}.
\end{remark}

\section{A priori estimates} \label{sec: a priori estimate}

The goal of this section is to estimate various norms of solutions to \eqref{main system} using constants depending only on the data.
In Section \ref{sec:basic a priori} we prove some standard results, including the usual ``energy" type estimate on the quantity $\int_0^T \int_0^L u_x^2 m \ dx dt$.
Then in Section \ref{sec:hamilton-jacobi} we prove a priori bounds on the solution to the Hamilton-Jacobi equation using classical techniques for parabolic equations.
Section \ref{sec:nonlocal term} is our most original contribution; there we show that the term $\int_0^L u_x(t,x)m(t,x)dx$ is a priori bounded uniformly in $t$.
Finally, in Section \ref{sec:full regularity} we use the previous estimates to prove higher regularity.

\subsection{Basic a priori estimates} \label{sec:basic a priori}
\begin{proposition}[Main a priori estimates] \label{main estimates}
Suppose $(u,m)$ is a pair of smooth functions satisfying \eqref{main system}.
Then
\begin{equation}
\label{positive integrable}
u(t,x) \geq 0, m(t,x) \geq 0, \|m(t)\|_{L^1(0,L)} \leq \|m_0\|_{L^1(0,L)} \ \forall t \in [0,T], \forall x \in [0,L].
\end{equation}
Moreover, for some $C > 0$ depending on the data,
we have
\begin{equation} \label{mu_x^2}
\int_0^T \int_0^L mu_x^2 dx dt \leq C,
\end{equation}
and
\begin{equation}
\label{mG^2}
\int_0^T \int_0^L m|G(t,u_x,[mu_x])|^2 dx dt = \int_0^T \int_0^L mH(t,u_x,[mu_x]) dx dt \leq C.
\end{equation}
\end{proposition}

\begin{proof}
We start by proving \eqref{positive integrable}.
Let $\phi$ be any function in $C^2(\bb{R})$ with $\phi'(0) = 0$.
Multiply \eqref{main system}(ii) by $\phi'(m)$ and integrate by parts to get
\begin{equation}
\int_0^L \phi(m)(t,x) dx = -\frac{1}{2}\sigma^2 \int_0^t \int_0^L \phi''(m)m_x^2 dx ds + \int_0^t \int_0^L \phi''(m)m_x Gm_- dx ds.
\end{equation}
Take $\phi(s) = (s_-)^{2+\delta}$ where $s_- := (|s|-s)/2$ and let $\delta \to 0$ to deduce
\begin{multline}
\int_0^L m_-(t,x)^{2} dx = -\sigma^2 \int_0^t \int_0^L |(m_-)_x|^2 dx ds + 2\int_0^t \int_0^L (m_-)_x Gm_- dx ds
\\
\leq -\frac{1}{2}\sigma^2 \int_0^t \int_0^L |(m_-)_x|^2 dx ds + \frac{2\|G\|_\infty^2}{\sigma^2}\int_0^t \int_0^L (m_-)^2 dx ds,
\end{multline}
where we note that $(m_0)_- \equiv 0$.
By Gronwall's inequality we obtain $m_-(t,x) \equiv 0$, which proves positivity.
On the other hand, if we take $\phi(s) = s^{1+\delta}$ and let $\delta \to 0$, then we deduce $\|m(t)\|_{L^1(0,L)} \leq \|m_0\|_{L^1(0,L)}$.

Also, since we have
\begin{equation}
\label{hamilton jacobi inequality}
-u_t - \frac{\sigma^2}{2}u_{xx} + ru \geq 0,
\end{equation}
we can deduce using similar arguments that $u \geq e^{-rT} \min_x u(T,x)$.
Hence, in particular, we have $u \geq 0$ by the assumption $u_T \geq 0$.
Thus we have proved \eqref{positive integrable}.

Next we prove \eqref{mu_x^2}, from which follows \eqref{mG^2}.
Multiply \eqref{main system}(i) by $m$ and \eqref{main system}(ii) by $u$ and integrate by parts to get
\begin{multline}
\int_0^L u_T(x)m(T,x) ~dx 
-\int_0^L u(0,x)m_0(x) ~dx 
= r\int_0^T \int_0^L um ~dxdt 
\\
- \int_0^T\int_0^L mH(t,u_x,[mu_x]) ~dxdt 
- \int_0^T\int_0^L mu_xG(t,u_x,[mu_x]) ~dxdt.
\end{multline}
Since $u \geq 0$ and $m \geq 0$, we get
\begin{equation} 
\int_0^L u_T(x)m(T,x) ~dx  + \int_0^T\int_0^L m[u_xG(t,u_x,[mu_x])+H(t,u_x,[mu_x])] ~dxdt  \geq 0,
\end{equation}
and then using the fact that $\|m(t)\|_{L^1} \leq 1$ we can rewrite this as
\begin{equation}
\label{local to nonlocal estimate}
\int_0^T\int_0^L mu_x^2 ~dxdt 
\leq \int_0^T\int_0^L m(a+c\bar{p})^2 ~dxdt + 4\|u_T\|_\infty
= \int_0^T \eta(t)(a+c\bar{p})^2 dt + 4\|u_T\|_\infty.
\end{equation}
since $a + c\bar{p}$ does not depend on $x$.

To analyze the right-hand side, we observe that
\begin{equation*}
a(\eta(t)) + c(\eta(t))\bar{p}(t) 
=
\frac{2}{2+\epsilon \eta(t)} + \frac{\epsilon}{2+\epsilon\eta(t)} \int_0^L u_x(t,x)m(t,x)dx.
\end{equation*}
By Cauchy-Schwartz we see that
\begin{equation} \label{a + cp estimate}
|a(\eta(t)) + c(\eta(t))\bar{p}(t)| \leq 1 + \frac{\epsilon \eta^{1/2}(t)}{2+\epsilon\eta(t)} \left(\int_0^L u_x^2(t,x)m(t,x)dx\right)^{1/2}.
\end{equation}
which implies, using the fact that $0 \leq \eta(t) \leq 1$,
\begin{align*}
\eta(t)(a(\eta(t)) + c(\eta(t))\bar{p}(t))^2 
&\leq (1+1/\delta)\eta(t) + (1+\delta)\frac{\epsilon^2 \eta^2(t)}{(2+\epsilon\eta(t))^2}\int_0^L u_x^2(t,x)m(t,x)dx
\\
&\leq 
1+ 1/\delta + \frac{(1+\delta)\epsilon}{2+\epsilon}\int_0^L u_x^2(t,x)m(t,x)dx
\end{align*}
for an arbitrary $\delta > 0$.
By choosing $\delta = 1/\epsilon$,
then \eqref{local to nonlocal estimate} becomes
$$
\int_0^T\int_0^L mu_x^2 ~dxdt \leq (2+\epsilon)(1+\epsilon)T + 4(2+\epsilon)\|u_T\|_\infty
$$ 
which yields \eqref{mu_x^2}.
As for \eqref{mG^2}, we combine \eqref{mu_x^2} with \eqref{a + cp estimate} and the definition of $G$ and $H$.
\end{proof}

We may now deduce certain a priori bounds on the Fokker-Planck equation, which will be useful later on.
\begin{lemma}[Regularity of $m$] \label{regularity of m}
Suppose $(u,m)$ is a pair of smooth functions satisfying \eqref{main system}.

Then there exists a constant $C > 0$ depending on the data such that
\begin{equation} \label{m_x^2}
\int_0^T \int_0^L \frac{m_x^2}{m+1}dx dt \leq C.
\end{equation}
\end{lemma}

\begin{proof}
Multiply \eqref{main system}(ii) by $\ln(m+1)$ and integrate by parts to get
\begin{multline}
\frac{d}{dt} \int_0^L \phi(m)(t) \ dx
= -\frac{\sigma^2}{2}\int_0^L \frac{m_x^2}{1+m}dx
- \int_0^L \frac{G m m_x}{1+m}dx
\\
\leq -\frac{\sigma^2}{4}\int_0^L \frac{m_x^2}{1+m}dx
+ \frac{1}{\sigma^2}\int_0^L \frac{G^2 m^2}{1+m}dx,
\end{multline}
where $\phi(m) = (1+m)\ln(1+m) - m$.
By Equation \ref{mG^2} in Proposition \ref{main estimates}, and using the fact that $m_0$ is bounded and $\phi(m) \geq 0$, we get
\begin{equation}
\frac{\sigma^2}{4}\int_0^T \int_0^L \frac{m_x^2}{1+m}dx 
\leq \int_0^L \phi(m_0)\ dx
+ \frac{1}{\sigma^2}\int_0^T \int_0^L G^2 m \ dx dt 
\leq C.
\end{equation}
\end{proof}


\subsection{A priori bounds for the Hamilton-Jacobi equation} \label{sec:hamilton-jacobi}

Let $f(t) := a(\eta(t)) + c(\eta(t))\bar{p}(t)$.
Then \eqref{main system}(i) reads as
\begin{equation} \label{parabolic equation}
u_t + \frac{\sigma^2}{2}u_{xx} - ru + \frac{1}{4}(f(t) - u_x)^2 = 0,
\end{equation}
from which we can estimate
\begin{equation}
-u_t - \frac{\sigma^2}{2}u_{xx} + ru \leq \frac{1}{2}f(t)^2 + \frac{1}{2}u_x^2.
\end{equation}
From Proposition \ref{main estimates} we know that $f \in L^2(0,T)$ with an a priori bound on its norm.
Using classical arguments, this is enough to infer an $L^\infty$ estimate on $u$ as well as an $L^2$ estimate on $u_x$, as the following proposition makes clear.
\begin{lemma} \label{a priori bounds on u}
Suppose $u$ is a smooth function on $[0,T] \times [0,L]$ satisfying
\begin{equation} \label{subsolution g}
-u_t - ku_{xx} \leq g(t) + ju_x^2, ~~ u(T,x) = u_T
\end{equation}
where $j,k$ are positive constants, $g(t) \geq 0$ is an integrable function on $[0,T]$, and $u_T = u_T(x)$ is a smooth function on $[0,L]$.
Assume that $u$ is bounded below, that $u(t,0) = 0$, and $u_x(t,L) = 0$.
Then there exists a constant $C = C(j,k,u_T,\|g\|_{L^1(0,T)})$ such that
\begin{equation}
\|u\|_\infty + \int_0^T \int_0^L u_x^2 \ dx dt \leq C.
\end{equation}
\end{lemma}

\begin{proof}
Let $w(t,x) = \exp\left\{\lambda\left(u(t,x) + \int_0^t g(s)ds\right)\right\} - 1$ for $\lambda = j/k$.
Then
\begin{equation} \label{heat subsolution}
-w_t - kw_{xx} \leq (j-k\lambda)\lambda e^{\lambda u}u_x^2 = 0.
\end{equation}
In particular $w$ satisfies the maximum principle, i.e.
$$
\max_{[0,T] \times [0,L]} w = \max_{\Gamma_T} w
$$
where $\Gamma_T = ([0,T] \times \{0\}) \cup (\{T\} \times [0,L])$.
To see this, it suffices to take $\mu = \max_{\Gamma_T} w$, then multiply \eqref{heat subsolution} by $(w-\mu)_+$ and integrate by parts (note that $w_x(t,L) = 0$) to get
$$
\int_0^L (w(t,x)-\mu)_+^2 dx = -k\int_0^T \int_0^L ((w-\mu)_+)_x^2 dx dt +  \int_0^L (w(T,x)-\mu)_+^2 dx \leq 0.
$$
It follows that $w \leq \mu$ everywhere,
from which we deduce that $u \leq \frac{1}{\lambda}\ln(\mu)$.
On the other hand, by the definition of $w$ we can directly compute
$$
\mu = \max_{x \in [0,L]}\exp\left\{\lambda\left(u_T(x) + \int_0^T g(t)dt\right)\right\} - 1,
$$
which is a constant depending only on $\|u_T\|_\infty$ and $\|g\|_{L^1(0,T)}$.

Using the same computation with $w$ instead of $(w-\mu)_+$ we get
$$
\int_0^T \int_0^L w_x^2 dx dt \leq \frac{1}{k}\int_0^L w(T,x)_+^2 dx \leq \frac{\mu^2 L}{k}.
$$
Since $w_x(t,x) = \lambda\exp\left\{\lambda\left(u(t,x) + \int_0^t g(s)ds\right)\right\}u_x(t,x)$ and $u$ is bounded below, we deduce
$$
\int_0^T \int_0^L u_x^2 \ dx dt  \leq \frac{\mu^2 L}{\lambda^2 k}e^{2\lambda\|u_-\|_\infty}.
$$
\end{proof}

\subsection{Analysis of nonlocal term} \label{sec:nonlocal term}

In order to obtain higher regularity on $u$, we need to analyze the nonlocal coupling term $\int_0^L u_x(t,x)m(t,x)dx$.
In particular, we will show it is bounded.

In the case when $\sigma = 0$, we have a fortuitous identity which follows from integration by parts.
Differentiate \eqref{hamilton-jacobi} to get
\begin{equation*}
u_{xt}  -ru_x - \frac{1}{2}(a+ c\bar{p} - u_x)u_{xx} = 0,
\end{equation*}
noting that $\sigma = 0$.
Then multiply by $m$ and \eqref{boltzmann} to get
\begin{equation*}
\frac{d}{dt} \int_0^L u_x(t,x)m(t,x)dx = r\int_0^L u_x(t,x)m(t,x)dx,
\end{equation*}
which means
\begin{equation*}
\int_0^L u_x(t,x)m(t,x)dx = e^{r(t-T)}\int_0^L u_x(T,x)m(T,x)dx.
\end{equation*}
Thus, as long as $u_T$ is smooth, we know that the term $\int_0^L u_x(t,x)m(t,x)dx$ is bounded uniformly in $t$.
This in turn implies that $\bar{p}(t)$ is bounded, which allows us to analyze the regularity of $u$ by classical methods for parabolic equations.

Unfortunately, when $\sigma > 0$ this formal calculation fails; we get instead
\begin{equation} \label{formal nonlocal identity}
e^{rt}\frac{d}{dt}e^{-rt} \int_0^L u_x(t,x)m(t,x)dx = - \frac{\sigma^2}{2}u_x(t,0)m_x(t,0) - \frac{\sigma^2}{2}u_{xx}(t,L)m(t,L)
\end{equation}
with no estimates on the boundary terms.
On the other hand, thanks to the following lemma, we note that each of the terms $u_x(t,0)m_x(t,0)$ and $u_{xx}(t,L)m(t,L)$ has a definite sign.
This will allow us to prove that each of these terms is integrable in time with an a priori bound on its $L^1(0,T)$ norm.
\begin{lemma} \label{u_x m_x sign}
Suppose $(u,m)$ is a pair of smooth functions satisfying \eqref{main system}.
Then $u_x(t,x) \geq 0$, $m_x(t,0) \geq 0$, and $u_{xx}(t,L) \leq 0$ for all $(t,x) \in [0,T] \times [0,L]$.
\end{lemma}

\begin{proof}
Notice that, by Proposition \ref{main estimates}, $u$ and $m$ are both non-negative, hence their minimum is attained at $u(t,0) = 0$ and $m(t,0) = 0$, respectively.
It follows that $u_x(t,0)$ and $m_x(t,0)$ are non-negative for all $t \in [0,T]$.

Differentiate \eqref{main system}(i) to get
\begin{equation} \label{differentiated hj}
u_{xt} + \frac{\sigma^2}{2}u_{xxx} -ru_x - \frac{1}{2}(a+ c\bar{p} - u_x)u_{xx} = 0.
\end{equation}
Note that $u_{x}(t,0)$ and $0 = u_x(t,L)$ are both non-negative.
As $u_T'(x)$ is also non-negative, it follows that $u_x \geq 0$ everywhere by a classical maximum principle argument.
Thus $u_x(t,L) = 0$ is a minimum of $u_x$, so $u_{xx}(t,L) \leq 0$.
\end{proof}

Returning to \eqref{formal nonlocal identity},
we note that the terms on the right-hand side have definite but opposite signs.
To show that each of them is integrable in time, we localize away from each boundary point.
Once this is accomplished, we can see that the term $\int_0^L u_x(t,x)m(t,x)dx$ is bounded.
We prove this in the following:
\begin{proposition}
\label{nonlocal term}
Suppose $(u,m)$ is a smooth solution of \eqref{main system}.

(i) For any $\delta > 0$, there exists a constant $C_\delta > 0$ such that
\begin{equation} \label{L-delta}
\left|\int_{0}^{L-\delta} u_x(t,x)m(t,x)dx \right| \leq C_{\delta} \ \ \forall t \in [0,T].
\end{equation}

(ii) For any $\delta > 0$, there exists a constant $C_\delta > 0$ such that
\begin{equation} \label{delta-L}
\left|\int_{\delta}^{L} u_x(t,x)m(t,x)dx \right| \leq C_{\delta} \ \ \forall t \in [0,T].
\end{equation}

(iii) By (i) and (ii), there exists a constant $C$ depending only on the data such that
\begin{equation} \label{nonlocal term estimate}
\left|\int_{0}^{L} u_x(t,x)m(t,x)dx \right| \leq C \ \ \forall t \in [0,T].
\end{equation}
\end{proposition}

\begin{proof}
Take a smooth, non-negative function $\zeta = \zeta(x)$ on $[0,L]$ to be further specified later.
Multiply \eqref{differentiated hj} by $\zeta m$ and integrate by parts using \eqref{main system}(ii) to get
\begin{multline} \label{cut-off}
e^{rt} \frac{d}{dt}\left(e^{-rt}\int_0^L u_x(t)m(t)\zeta~ dx\right)
= -\frac{\sigma^2}{2}\zeta(L)u_{xx}(t,L)m(t,L) - \frac{\sigma^2}{2}\zeta(0)u_x(t,0)m_x(t,0)
\\
-\sigma^2\int_0^L \zeta_x u_x m_x \ dx
 - \frac{\sigma^2}{2}\int_0^L \zeta_{xx} u_x m \ dx 
- \frac{1}{2}(a+c\bar{p})\int_0^L \zeta_x u_x m \ dx
+ \frac{1}{2}\int_0^L \zeta_x u_x^2 m \ dx.
\end{multline}
Let us estimate the time integral of each of the terms in the second line of \eqref{cut-off}.
First we have
\begin{equation}
\int_0^T \left| \sigma^2\int_0^L \zeta_x u_x m_x \ dx \right| dt
\leq \sigma^2\|\zeta_x\|_\infty \int_0^T \int_0^L\left\{ u_x^2 (m+1) + \frac{m_x^2}{m+1}\right\} dx dt
\leq C\|\zeta_x\|_\infty.
\end{equation}
using Proposition \ref{main estimates} and Lemma \ref{a priori bounds on u}.
Likewise,
\begin{equation}
\int_0^T \left| \frac{\sigma^2}{2}\int_0^L \zeta_{xx} u_x m \ dx \ dx \right| dt
\leq \frac{\sigma^2}{2}\|\zeta_{xx}\|_\infty \int_0^T \int_0^L\left\{ (u_x^2 + 1)m\right\} dx dt
\leq C\|\zeta_{xx}\|_\infty,
\end{equation}
\begin{equation}
\int_0^T \left| (a+c\bar{p})\int_0^L \zeta_x u_x m \ dx \right| dt
\leq \|\zeta_x\|_\infty \int_0^T \left\{(a+c\bar{p})^2 + \int_0^L u_x^2 m \ dx\right\} dt
\leq C\|\zeta_x\|_\infty,
\end{equation}
and finally
\begin{equation}
\int_0^T \left| \int_0^L \zeta_x u_x^2 m \ dx \right| dt
\leq \|\zeta_x\|_\infty \int_0^T \int_0^L u_x^2 m \ dx dt
\leq C\|\zeta_x\|_\infty.
\end{equation}
To summarize, we may write \eqref{cut-off} as
\begin{equation} \label{cut-off1}
e^{rt} \frac{d}{dt}\left(e^{-rt}\int_0^L u_x(t)m(t)\zeta~ dx\right)
= -\frac{\sigma^2}{2}\zeta(L)u_{xx}(t,L)m(t,L)
 - \frac{\sigma^2}{2}\zeta(0)u_x(t,0)m_x(t,0)
+ I_\zeta(t),
\end{equation}
where
$$
\int_0^T |I_\zeta(t)|dt \leq C(\zeta)
$$
such that $C(\zeta)$ depends only on $\|\zeta_x\|_\infty,\|\zeta_{xx}\|_\infty$, and previous estimates.

Now let us prove (i).
We specify that $\zeta(x) = 0$ for $L - \delta/2 \leq x \leq L$, $\zeta(x) = 1$ for $0 \leq x \leq L - \delta$, and $0 \leq \zeta \leq 1$.
Then we can assume $C(\zeta) \leq C_\delta$, where $C_\delta$ is some constant proportional to $1/\delta^2$ for $\delta > 0$ small.
Integrating \eqref{cut-off1} over $[0,T]$ we get
\begin{multline}
\frac{\sigma^2}{2}\int_0^T e^{-rt} u_x(t,0)m_x(t,0)dt
\\
 = \int_0^L u_x(0,x)m_0(x)\zeta~ dx - e^{-rT} \int_0^L u_T'(x)m(T,x)\zeta~ dx
+ \int_0^T e^{-rt}I_\zeta(t)dt.
\end{multline}
Now on the one hand, using the fact that $u_x \geq 0$ and that $u$ is bounded (Lemma \ref{a priori bounds on u}), we have
\begin{equation}
\int_0^L |u_x(0,x)m_0(x)\zeta|~ dx = \int_0^L u_x(0,x)m_0(x)\zeta~ dx \leq \|m_0\|_\infty \int_0^L u_x(0,x) dx = \|m_0\|_\infty u(0,L) \leq C.
\end{equation}
On the other hand, since $\int_0^L m(t,x)dx \leq 1$ for all $t$, we have
\begin{equation}
\int_0^L |u_T'(x)m(T,x)\zeta|~ dx \leq \|u_T'\|_\infty.
\end{equation}
Using the fact that $u_x \geq 0$ and $m_x(t,0) \geq 0$ from Lemma \ref{u_x m_x sign}, we deduce that
\begin{equation}
\int_0^T |u_x(t,0)m_x(t,0)|dt 
\leq e^{rT}\int_0^T e^{rt} u_x(t,0)m_x(t,0)dt \leq C_\delta.
\end{equation}
Finally, integrating \eqref{cut-off1} this time over $[t,T]$ we get
\begin{multline}
\int_0^L u_x(t,x)m(t,x)\zeta~ dx 
\\
= \frac{\sigma^2}{2}\int_t^T e^{-rt} u_x(t,0)m_x(t,0)dt + e^{-rT} \int_0^L u_T'(x)m(T,x)\zeta~ dx - \int_t^T e^{-rt}I_\zeta(t)dt
\end{multline}
from which we obtain \eqref{L-delta}.

In a similar way we can prove (ii).
This time we specify that $\zeta(x) = 0$ for $0 \leq x \leq \delta/2$, $\zeta(x) = 1$ for $\delta \leq x \leq L$, and $0 \leq \zeta \leq 1$.
Again we can assume $C(\zeta) \leq C_\delta$, where $C_\delta$ is some constant proportional to $1/\delta^2$ for $\delta > 0$ small.
Integrating \eqref{cut-off1} over $[0,T]$ we get
\begin{multline}
-\frac{\sigma^2}{2}\int_0^T e^{-rt} u_{xx}(t,L)m(t,L)dt
\\
 = -\int_0^L u_x(0,x)m_0(x)\zeta~ dx + e^{-rT} \int_0^L u_T'(x)m(T,x)\zeta~ dx
- \int_0^T e^{-rt}I_\zeta(t)dt.
\end{multline}
Now since $u_{xx}(t,L) \leq 0$ from Lemma \ref{u_x m_x sign} and $m \geq 0$, we can deduce 
\begin{equation}
\int_0^T |u_{xx}(t,L)m(t,L)| dt \leq C_\delta
\end{equation}
using the same estimates as in the proof of part (i).
Now integrating \eqref{cut-off1} over $[t,T]$ we get
\begin{multline}
\int_0^L u_x(t,x)m(t,x)\zeta~ dx 
\\
= \frac{\sigma^2}{2}\int_t^T e^{-rt} u_{xx}(t,L)m(t,L)dt + e^{-rT} \int_0^L u_T'(x)m(T,x)\zeta~ dx - \int_t^T e^{-rt}I_\zeta(t)dt
\end{multline}
from which we infer \eqref{delta-L}.

Finally, \eqref{nonlocal term estimate} follows from \eqref{L-delta} and \eqref{delta-L} by fixing any $\delta < L/2$.
\end{proof}

\begin{corollary} \label{f(t) bounded}
Suppose $(u,m)$ is a smooth solution of \eqref{main system}.
Then for some constant $C$ depending only on the data,
\begin{equation} \label{a + cp bounded}
|a(\eta(t)) + c(\eta(t))\bar{p}(t)| \leq C \ \ \forall t \in [0,T].
\end{equation}
\end{corollary}

\begin{proof}
By the definition of $a,c,$ and $\bar{p}$, it is enough to have $|\int_0^T u_x(t,x)m(t,x)dx| \leq C$ for all $t \in [0,T]$, which we get from Proposition \ref{nonlocal term}.
\end{proof}

\subsection{Full regularity of $u$} \label{sec:full regularity}

Let us return to \eqref{parabolic equation}:
\begin{equation*} 
u_t + \frac{\sigma^2}{2}u_{xx} - ru + \frac{1}{4}(f(t) - u_x)^2 = 0,
\end{equation*}
where we recall $f(t) := a(\eta(t)) + c(\eta(t))\bar{p}(t)$.
We can now write
\begin{equation} \label{parabolic estimate}
-u_t - \frac{\sigma^2}{2}u_{xx} + ru \leq C_1 + \frac{1}{2}u_x^2,
\end{equation}
where $C_1$ is the constant coming from Corollary \ref{f(t) bounded}.
It is now possible to obtain global estimates on $|u_x|$, which we will then be able to ``bootstrap" to gain higher regularity on $u$.

\begin{lemma} \label{C^1 bounds}
For $(u,m)$ a smooth solution of \eqref{main system}, we have
\begin{equation} \label{u_x bounded}
|u_x(t,x)| \leq C \ \ \forall (t,x) \in [0,T] \times [0,L].
\end{equation}
for some constant $C$ depending only on the data.
\end{lemma}

\begin{proof}
By Lemma \ref{u_x m_x sign} we already know $u_x \geq 0$.
It suffices to prove that $u_x \leq C$.
We claim
\begin{equation}
u_x(t,0) \leq C.
\end{equation}
To see this, set $v = e^{u/\sigma^2} - 1$ and use \eqref{parabolic estimate} to get
\begin{equation}
-v_t - \frac{\sigma^2}{2}v_{xx} \leq C_1\frac{1}{\sigma^2}e^{\|u\|_\infty/\sigma^2} =: \tilde{C}_1.
\end{equation}
Note that $v(t,0) = v_x(t,L) = 0$.
Set $\tilde{v} = v + Me^{-x}$ where $M$ is large enough that
$$
v_x(T,x) = \frac{1}{\sigma^2}e^{u_T(x)/\sigma^2}u_T'(x) \leq Me^{-L}
$$
and
$$
\tilde{C}_1 \leq \frac{\sigma^2}{2}Me^{-L}.
$$
Then, on the one hand, we have $-\tilde{v}_t - \frac{\sigma^2}{2} \tilde{v}_{xx} \leq 0$, and since $\tilde{v}_x(t,L) = -Me^{-L} \leq 0$ we get as before the maximum principle
$$
\max_{[0,T] \times [0,L]} \tilde{v} \leq \max_{\Gamma_T} \tilde{v}.
$$
On the other hand, we have $\tilde{v}_x(T,x) \leq 0$ and so $\tilde{v}(T,x) \leq \tilde{v}(T,0) = M$ for all $x \in [0,L]$.
It follows that $\tilde{v}(t,0) = M$ is the global maximum of $\tilde{v}$, hence $\tilde{v}_x(t,0) \leq 0$ at each $t \in [0,T]$.
Recalling the definition of $\tilde{v}$ we get
$$
\frac{1}{\sigma^2}e^{u/\sigma^2}u_x(t,0) \leq M,
$$
and since $u \geq 0$ we get $u_x(t,0) \leq M\sigma^2$, which is the desired estimate.

Taking into account the assumption that $u_T$ is smooth, we have thus shown
\begin{equation}
\max_{\Gamma_T} |u_x| \leq C, \ \ \Gamma_T := ([0,T] \times \{0,L\}) \cup (\{T\} \times [0,L])
\end{equation}
where $C$ depends only on the data.

Now differentiate \eqref{parabolic equation} to get
\begin{equation} \label{u_x equation}
-u_{xt} - \frac{\sigma^2}{2}u_{xxx} + ru_x + \frac{1}{2}(f-u_x)u_{xx} = 0.
\end{equation}
Then $w(t,x) = u_x(t,x)e^{-rt}$ satisfies
\begin{equation}
-w_t - \frac{\sigma^2}{2}w_{xx} +  \frac{1}{2}(f-u_x)w_{x} = 0.
\end{equation}
By classical arguments, $w$ satisfies the maximum principle, i.e.
$$
|w(t,x)| \leq \max_{\Gamma_T} |w| \leq \max_{\Gamma_T} |u_x|,
$$
from which it follows that
$$
|u_x(t,x)| \leq e^{rT}\max_{\Gamma_T} |u_x| \leq C.
$$
\end{proof}

Corollary \ref{C^1 bounds} permits us to obtain higher order estimates for $u$.
\begin{proposition} \label{C^2 estimates}
There exists a constant $C$ depending only on the data such that if $(u,m)$ is a smooth solution of \eqref{main system}, then for some $\alpha > 0$
\begin{equation} \label{C^{1,2} bound}
\|u\|_{C^{1+\alpha/2,2+\alpha}(\overline{Q_T})} + \|m\|_{C^{1+\alpha/2,2+\alpha}(\overline{Q_T})} \leq C
\end{equation}
where $C^{1+\alpha/2,2+\alpha}(\overline{Q_T})$ is the parabolic H\"older space defined on $\overline{Q_T} = [0,T] \times [0,L]$.
\end{proposition}

\begin{proof}
Observe that $u$ and $u_x$ each satisfy a linear parabolic equation with coefficients which are bounded by constants depending on the data.
By \cite[Theorem IV.9.1]{ladyzhenskaia1968linear}, we have that $u$ is bounded in $L^p(0,T;W^{2,p}(0,L)) \cap W^{1,p}(Q_T)$ for any $p > 1$ and thus $u,u_x$ and $u_{xx}$ are all bounded in a H\"older space $C^\beta(\overline{Q_T})$ for some $\beta > 0$.
Then by \cite[Theorem III.7.1]{ladyzhenskaia1968linear} $m$ has an a priori bound in $L^\infty(Q_T)$.
Further, we observe that \eqref{main system}(ii) can be written
\begin{equation}
m_t - \frac{\sigma^2}{2}m_{xx} - \frac{1}{2}(a+c\bar{p}-u_x)m_x + \frac{1}{2}u_{xx}m = 0,
\end{equation}
which also has coefficients bounded by the data.
Using the same technique as in Lemma \ref{C^1 bounds} we obtain an a priori estimate on $m_x(0,x)$.
Then \eqref{formal nonlocal identity} can be used to directly estimate the term $\int_0^L u_x(t,x)m(t,x)dx$ in $C^1([0,T])$.
We can now see that \eqref{main system}(i),(ii) are both parabolic equations with coefficients estimated in H\"older spaces by constants depending only on the data.
Applying \cite[Theorems IV.5.2 and IV.5.3]{ladyzhenskaia1968linear} now gives the conclusion.
\end{proof}

\section{Existence} \label{sec:existence}

We now prove the main result of this paper.
\begin{theorem}
There exists a classical solution of \eqref{main system}.
\end{theorem}

\begin{proof}
We use the Leray-Schauder fixed point theorem.
Consider the map $(u,m) \mapsto (v,f) = T(u,m;\tau)$ given by solving the following parametrized set of PDE systems:
\begin{equation}
\label{parametrized system}
\left\{
\begin{array}{rcc}
(i) & v_t + \frac{1}{2}\sigma^2 v_{xx} - rv + \tau H(t,u_x,[mu_x]) = 0, & 0 < t < T, 0 < x < L\\
(ii) & f_t - \frac{1}{2}\sigma^2 f_{xx} - \tau \left(G(t,v_x,[mv_x])f\right)_x = 0, & 0 < t < T, 0 < x < L\\
(iii) & f(0,x) = \tau m_0(x), v(T,x) = \tau u_T(x), & 0 \leq x \leq L\\
(iv) & v(t,0) = f(t,0) = 0, ~~ v_x(t,L) = 0, & 0 \leq t \leq T \\ 
(v) & \frac{1}{2}\sigma^2 f_x(t,L) + \tau G(t,v_x(t,L),[mv_x])f(t,L) = 0, & 0 \leq t \leq T
\end{array}\right.
\end{equation}
for $\tau \in [0,1]$.
Let $X$ be the space of all $(u,m)$ such that $u$ and $u_x$ are both H\"older continuous, say in $C^{\delta}(\overline{Q_T})$,  and $m$ is in $W^{1,\infty}(0,T;L^1(0,L))$.
Then by inspecting the definitions of $G$ \eqref{boltzmann} and $H$ \eqref{hamilton-jacobi} we find that $H(t,u_x,[mu_x])$ and $G(t,u_x,[mu_x])$ are both H\"older continuous as well.
By \cite[Theorem IV.5.2]{ladyzhenskaia1968linear} there is a solution $v$ of \eqref{parametrized system}(i) satisfying (iii) and (iv) such that $v \in C^{1+\alpha/2,2+\alpha}(\overline{Q_T})$ for some $\alpha > 0$ and 
\begin{equation}
\|v\|_{C^{1+\alpha/2,2+\alpha}(\overline{Q_T})} \leq C\|(u,m)\|_{X} = C(\|u\|_{C^{\alpha/2,\alpha}(\overline{Q_T})} + \|u_x\|_{C^{\alpha/2,\alpha}(\overline{Q_T})} + \|m\|_{W^{1,\infty}(0,T;L^1(0,L))})
\end{equation}
for some generic constant $C$.
Next, we write \eqref{parametrized system}(ii) as
$$
f_t - \frac{1}{2}\sigma^2 f_{xx} - \tau G(t,v_x,[mv_x])f_x + \frac{\tau}{2}v_{xx} f = 0,
$$
and we note that the coefficients are H\"older continuous.
Furthermore, because $m \in W^{1,\infty}(0,T;L^1(0,L))$ we can see that
$$
G(t,v_x(t,L),[mv_x]) = \frac{2}{2+\epsilon \eta(t)} + \frac{\epsilon}{2+\epsilon \eta(t)}\int_0^L v_x(t,y)m(t,y)dy
$$
is independent of $x$ and has a bounded time derivative.
Then we can apply \cite[Theorem IV.5.3]{ladyzhenskaia1968linear} to get a solution $f \in C^{1+\alpha/2,2+\alpha}(\overline{Q_T})$ satisfying \eqref{parametrized system}(iii) and (v) such that 
\begin{equation}
\|f\|_{C^{1+\alpha/2,2+\alpha}(\overline{Q_T})} \leq C(\|v\|_{C^{1+\alpha/2,2+\alpha}(\overline{Q_T})} + \|m\|_{W^{1,\infty}(0,T;L^1(0,L))} \leq C\|(u,m)\|_{X}.
\end{equation}
It follows that $T : X \times [0,1] \to X$ is a well-defined and compact mapping, since $C^{1+\alpha/2,2+\alpha}(\overline{Q_T}) \times C^{1+\alpha/2,2+\alpha}(\overline{Q_T})$ is compact in $X$ by the Arzel\`a-Ascoli Theorem.

Now for $\tau = 0$ we have $T(u,m;0) = 0$ for all $(u,m)$ by standard theory for the linear heat equation.
On the other hand, if $(u,m) \in X, \tau \in [0,1]$ is such that $(u,m) = T(u,m;\tau)$ then $(u,m) \in C^{1+\alpha/2,2+\alpha}(\overline{Q_T}) \times C^{1+\alpha/2,2+\alpha}(\overline{Q_T})$ is a solution of \eqref{main system} with $m_0,u_T$, $G$ and $H$ replaced by $\tau m_0, \tau u_T,$ $\tau G$ and $\tau H$, respectively.
Then the a priori estimates of Section \ref{sec: a priori estimate} carry through uniformly in $\tau \in [0,1]$, and so by Proposition \ref{C^2 estimates} we obtain a constant $C_0$ depending only on the data such that
$$
\|(u,m)\|_{X} \leq C_0.
$$
By the Leray-Schauder fixed point theorem (see e.g.~\cite[Theorem 10.6]{gilbarg2015elliptic}), there exists $(u,m)$ such that $(u,m) = T((u,m);1)$, which is a classical solution to \eqref{main system}.
\end{proof}

\section{Uniqueness} \label{sec:uniqueness}

The structure of \eqref{main system} makes uniqueness a nontrivial issue.
Unlike traditional mean field games in which uniqueness is verified by a straightforward use of the ``energy identity" (thanks to the fact that the coupling is monotone \cite{lasry07}),
System \eqref{main system} does not allow such an argument. 
Our uniqueness result will rely heavily on the fact that solutions are {\em smooth} with a priori bounds, and it will hold only when $\epsilon$ is small.

We now proceed to state and prove the main uniqueness result.
\begin{theorem}
There exists $\epsilon_0 > 0$ sufficiently small such that for any $\epsilon \leq \epsilon_0$, \eqref{main system} has at most one classical solution.
\end{theorem}

\begin{proof}
Let $(u_1,m_1)$ and $(u_2,m_2)$ be two solutions.
Define, for $i=1,2$,
$$
H_i = \frac{1}{4}(a(\eta_i(t)) + c(\eta_i(t))\bar{p}_i(t) - \partial_x u_i(t,x))^2, 
\ \
G_i = \frac{1}{2}(a(\eta_i(t)) + c(\eta_i(t))\bar{p}_i(t) - \partial_x u_i(t,x))
$$
where $\eta_i(t)$ and $\bar{p}_i(t)$ are defined according to the definitions in \eqref{average} and \eqref{average price}, with $u$ and $m$ replaced by $u_i$ and $m_i$.
Then, in particular, $u= u_1 - u_2$ satisfies
\begin{equation}
\label{hj differences}
u_t + \frac{\sigma^2}{2} u_{xx} - ru + H_1 - H_2 = 0, \ \ u(T,\cdot) \equiv 0
\end{equation}
while $m = m_1 - m_2$ satisfies
\begin{equation}
\label{fp differences}
m_t - \frac{\sigma^2}{2} m_{xx} - (G_1m_1 - G_2m_2)_x = 0, \ \ m(0,\cdot) \equiv 0.
\end{equation}
Let us introduce some notation.
Observe that
$$
a(\eta_i(t)) + c(\eta_i)(t))\bar{p}_i(t) = \frac{2}{2+\epsilon\eta_i(t)} + \frac{\epsilon}{2+\epsilon\eta_i(t)}\int_0^L \partial_x u_i(t,x)m_i(t,x)dx.
$$
With this in mind, we define
$$
A_i(t) = \frac{2}{2+\epsilon\eta_i(t)},
\ \ \
B_i(t) = \frac{\epsilon}{2+\epsilon\eta_i(t)}\int_0^L \partial_x u_i(t,x)m_i(t,x)dx.
$$
Notice that $2G_i = A_i + B_i - \partial_x u_i$ and $H_i = G_i^2$.

Now multiply \eqref{hj differences} by $m_1-m_2$ and \eqref{fp differences} by $u_1-u_2$, integrate by parts and add to get the typical energy identity for mean field games:
\begin{equation}
\label{energy difference 1}
\int_0^T \int_0^L e^{-rt}[(H_1-H_2)(m_1-m_2) + (G_1m_1-G_2m_2)(\partial_x u_1 - \partial_x u_2)] dx dt = 0,
\end{equation}
which can be rearranged to get
\begin{multline}
\label{energy difference 2}
\int_0^T \int_0^L e^{-rt}m_1[H_1-H_2 + (\partial_x u_1 - \partial_x u_2)G_1]dx dt
\\
+ \int_0^T \int_0^L e^{-rt}m_2[H_2-H_1 + (\partial_x u_2 - \partial_x u_1)G_2]dx dt
= 0.
\end{multline}
This, in turn, can be rearranged to give
\begin{multline}
\label{energy difference 3}
-\int_0^T e^{-rt}\left(A_1(t)^2-A_2(t)^2\right)(\eta_1(t)-\eta_2(t))dt
\\
+ \int_0^T \int_0^L e^{-rt} (m_1+m_2)(\partial_x u_1 - \partial_x u_2)^2 dx dt
= I_1 + I_2
\end{multline}
where
\begin{equation}
I_1 := \int_0^T e^{-rt}(\eta_1(t)-\eta_2(t))\{B_1(t)^2 - B_2(t)^2
+ 2A_1(t)B_1(t) - 2A_2(t)B_2(t) \}dt
\end{equation}
and
\begin{equation}
I_2 := 2\int_0^T \int_0^L e^{-rt} (m_2\partial_x u_1 - m_1\partial_x u_2)(A_1(t)+B_1(t)-A_2(t)-B_2(t)) dx dt.
\end{equation}
By a simple computation we have
\begin{equation}
-\int_0^T e^{-rt}\left(A_1(t)^2-A_2(t)^2\right)(\eta_1(t)-\eta_2(t))dt
\geq \frac{8\epsilon}{(1+\epsilon)^3}\int_0^T e^{-rt} (\eta_1(t)-\eta_2(t))^2 dt
\end{equation}
so we can write
\begin{equation} \label{uniqueness1}
8\epsilon\int_0^T e^{-rt} (\eta_1(t)-\eta_2(t))^2 dt + \int_0^T \int_0^L e^{-rt} (m_1+m_2)(\partial_x u_1 - \partial_x u_2)^2 dx dt
\leq I_1 + I_2.
\end{equation}
Our main task is to estimate $I_1$ and $I_2$.
For this we will use the a priori estimates from Section \ref{sec: a priori estimate}, which say in particular that $|\partial_x u_i| \leq C_0$ for some $C_0$ depending only on the data and on $\epsilon_0$, where $\epsilon \leq \epsilon_0$.
Since we are going to make $\epsilon$ small, we need to keep in mind that the constant $C_0$ does not change for decreasing values of $\epsilon$.

First, we estimate $|B_1(t)-B_2(t)|$.
We have
\begin{multline}
B_1(t) - B_2(t) =
\\
\frac{\epsilon}{2+\epsilon\eta_1(t)}\int_0^L \partial_x u_1(t,x)m_1(t,x)dx
- \frac{\epsilon}{2+\epsilon\eta_2(t)}\int_0^L \partial_x u_2(t,x)m_2(t,x)dx
\\
= \left(\frac{\epsilon}{2+\epsilon\eta_1(t)} - \frac{\epsilon}{2+\epsilon\eta_2(t)}\right)\int_0^L \partial_x u_1(t,x)m_1(t,x)dx
\\
+ \frac{\epsilon}{2(2+\epsilon\eta_2(t))}\int_0^L \partial_x (u_1(t,x)-u_2(t,x))(m_1(t,x)+m_2(t,x))dx
\\
+
\frac{\epsilon}{2(2+\epsilon\eta_2(t))}\int_0^L (\partial_x u_1(t,x) + \partial_x u_2(t,x))(m_1(t,x) - m_2(t,x))dx.
\end{multline}
Observe that
\begin{align*}
\left|\frac{1}{2+\epsilon\eta_1(t)} - \frac{1}{2+\epsilon\eta_2(t)}\right|
&\leq 
\frac{\epsilon|\eta_1(t)-\eta_2(t)|}{4},
\\
\int_0^L|m_1(t,x)-m_2(t,x)| dx
&\leq L^{1/2}\left(\int_0^L |m_1(t,x)-m_2(t,x)|^2dx\right)^{1/2},  \ \ \text{and}
\\
\int_0^L (\partial_x u_1 + \partial_x u_2)(m_1 - m_2)dx,
&\leq \sqrt{2}\left(\int_0^L (\partial_x u_1 + \partial_x u_2)^2(m_1 - m_2)dx\right)^{1/2}
\end{align*}
by the Cauchy-Schwartz inequality.
So using the uniform pointwise bounds on $u_x(t,x)$,
we obtain
\begin{multline} \label{B_1-B_2}
|B_1(t) - B_2(t)|
\leq \frac{\epsilon}{2\sqrt{2}}\left(\int_0^L (\partial_x u_1 - \partial_x u_2)^2(m_1+m_2) \ dx\right)^{1/2}
\\
+ \frac{\epsilon^2}{4}|\eta_1(t)-\eta_2(t)|
+ \frac{\epsilon}{4}C_0 L^{1/2}\left(\int_0^L |m_1-m_2|^2dx\right)^{1/2}.
\end{multline}
On the other hand, we have $|B_i(t)| \leq C_0 \epsilon/2$ for $i=1,2$.
We deduce from \eqref{B_1-B_2} that
\begin{multline} \label{B^2}
|B_1(t)^2 - B_2(t)^2| \leq \frac{\epsilon^2}{2\sqrt{2}} C_0\left(\int_0^L (\partial_x u_1 - \partial_x u_2)^2(m_1+m_2) \ dx\right)^{1/2}
\\
+ \frac{\epsilon^3}{4}C_0|\eta_1(t)-\eta_2(t)|
+ \frac{\epsilon^2}{4}C_0^2 L^{1/2}\left(\int_0^L |m_1-m_2|^2dx\right)^{1/2}.
\end{multline}
We also have
\begin{equation} \label{A_1-A_2}
|A_1(t)-A_2(t)| \leq \frac{\epsilon}{2}|\eta_2(t)-\eta_1(t)|.
\end{equation}
as well as $|A_i(t)| \leq 1$ for $i=1,2$.
Combine this with \eqref{B_1-B_2} to get
\begin{multline} \label{AB}
|A_1(t)B_1(t)-A_2(t)B_2(t)| \leq \frac{\epsilon}{2\sqrt{2}}\left(\int_0^L (\partial_x u_1 - \partial_x u_2)^2(m_1+m_2) \ dx\right)^{1/2}
\\
+ \frac{\epsilon^2}{4}(1+C_0)|\eta_1(t)-\eta_2(t)|
+ \frac{\epsilon}{4}C_0 L^{1/2}\left(\int_0^L |m_1-m_2|^2dx\right)^{1/2}.
\end{multline}
From \eqref{B^2} and \eqref{AB} it follows that
\begin{multline} \label{I_1}
I_1 \leq 
 \frac{\epsilon^2 C_0 + \epsilon}{2\sqrt{2}}\int_0^T e^{-rt}(\eta_1(t)-\eta_2(t))\left(\int_0^L (\partial_x u_1 - \partial_x u_2)^2(m_1+m_2) \ dx\right)^{1/2} dt
\\
+ \frac{\epsilon^3 C_0 + \epsilon^2(1+C_0)}{4}\int_0^T e^{-rt}(\eta_1(t)-\eta_2(t))^2 dt
+
\\
\frac{\epsilon^2C_0^2 + \epsilon C_0}{4}L^{1/2}\int_0^T e^{-rt}(\eta_1(t)-\eta_2(t))\left(\int_0^L |m_1-m_2|^2dx\right)^{1/2}dt
\\
\leq P_1(\epsilon)\int_0^T  \int_0^L e^{-rt}(m_1+m_2)(\partial_x u_1 - \partial_x u_2)^2 dx dt
+ P_2(\epsilon)\int_0^T \int_0^L e^{-rt}(m_2-m_1)^2 dx dt
\end{multline}
where $P_1(\epsilon),P_2(\epsilon) \to 0$ as $\epsilon \to 0$.
As for $I_2$, setting $D(t) = A_1(t)+B_1(t)-A_2(t)-B_2(t)$, we write
\begin{multline}
I_2 = \int_0^T \int_0^L e^{-rt}D(t) (m_2-m_1)(\partial_x u_1 + \partial_x u_2) dx dt
\\
+ \int_0^T \int_0^L e^{-rt}D(t) (m_1+m_2)(\partial_x u_1 - \partial_x u_2) dx dt
\\
\leq
2C_0L^{1/2}\int_0^T e^{-rt}|D(t)| \left(\int_0^L (m_2-m_1)^2 dx\right)^{1/2} dt
\\
+ \sqrt{2}\int_0^T  e^{-rt}|D(t)| \left(\int_0^L (m_1+m_2)(\partial_x u_1 - \partial_x u_2)^2\right)^{1/2} dx dt.
\end{multline}
By \eqref{B_1-B_2} and \eqref{A_1-A_2} we have
\begin{equation} \label{I_2}
I_2 
\leq
P_3(\epsilon)\int_0^T \int_0^L e^{-rt}(m_2-m_1)^2 dx dt
+ P_4(\epsilon)\int_0^T  \int_0^L e^{-rt}(m_1+m_2)(\partial_x u_1 - \partial_x u_2)^2 dx dt
\end{equation}
where $P_3(\epsilon),P_4(\epsilon) \to 0$ as $\epsilon \to 0$.
By \eqref{I_1} and \eqref{I_2}, Equation \eqref{uniqueness1} becomes
\begin{equation}
\label{uniqueness2}
\int_0^T  \int_0^L e^{-rt}(m_1+m_2)(\partial_x u_1 - \partial_x u_2)^2 dx dt
\leq \frac{P_2(\epsilon) + P_3(\epsilon)}{1-P_1(\epsilon)-P_4(\epsilon)}\int_0^T \int_0^L e^{-rt}(m_2-m_1)^2 dx dt.
\end{equation}

Now we consider the Fokker-Planck equation \eqref{fp differences}.
Recall that $m = m_1-m_2$.
Multiply by $m$ and integrate by parts to get
\begin{multline}
\frac{1}{2}\int_0^L m^2(t,x)dx = -\frac{\sigma^2}{2}\int_0^t \int_0^L m_x^2 \ dx dt
- \int_0^t \int_0^L m_x(G_1m_1-G_2m_2) dx dt
\\
\leq -\frac{\sigma^2}{4}\int_0^t \int_0^L m_x^2 \ dx dt
+ \frac{1}{\sigma^2}\int_0^t \int_0^L (G_1m_1-G_2m_2)^2 dx dt
\\
\leq \frac{2}{\sigma^2}\int_0^t \int_0^L G_1^2(m_1-m_2)^2 dx dt
+ \frac{2}{\sigma^2}\int_0^t \int_0^L |G_1-G_2|^2 m_2^2 \ dx dt.
\end{multline}
Recall that $G_i = \frac{1}{2}(A_i+B_i - \partial_x u_i)$.
Then, using \eqref{B_1-B_2},\eqref{A_1-A_2}, and the fact that $G_1$ and $m_2$ are bounded by some constant depending on the data, we obtain
\begin{equation}
\int_0^L (m_1(t,x)-m_2(t,x))^2dx
\leq C\int_0^t \int_0^L (m_1-m_2)^2 dx dt
+ P_5(\epsilon)\int_0^t \int_0^L (\partial_x u_1-\partial_x u_2)^2 m_2 dx dt
\end{equation}
where $P_5(\epsilon) \to 0$ as $\epsilon \to 0$.
By Gronwall's Lemma, we have
\begin{equation}
\sup_{t \in [0,T]} \int_0^L (m_1(t,x)-m_2(t,x))^2 dx \leq e^{CT}P_5(\epsilon)\int_0^T \int_0^L (\partial_x u_1-\partial_x u_2)^2 m_2 dx dt,
\end{equation}
and then by appealing to \eqref{uniqueness2} we have
\begin{multline} \label{uniqueness3}
\sup_{t \in [0,T]} \int_0^L (m_1(t,x)-m_2(t,x))^2 dx \leq \frac{e^{CT}P_5(\epsilon)(P_2(\epsilon) + P_3(\epsilon))}{1-P_1(\epsilon)-P_4(\epsilon)}\int_0^T \int_0^L e^{-rt}(m_2-m_1)^2 dx dt
\\
\leq \frac{T e^{CT}P_5(\epsilon)(P_2(\epsilon) + P_3(\epsilon))}{1-P_1(\epsilon)-P_4(\epsilon)}\sup_{t \in [0,T]} \int_0^L (m_1(t,x)-m_2(t,x))^2 dx.
\end{multline}
Fix $\epsilon$ small enough;
then \eqref{uniqueness3} implies
\begin{equation}
\sup_{t \in [0,T]} \int_0^L (m_1(t,x)-m_2(t,x))^2 dx = 0,
\end{equation}
i.e.~$m_1 = m_2$.
Returning to \eqref{uniqueness2} we also have
\begin{equation}
\int_0^T  \int_0^L (m_1+m_2)(\partial_x u_1 - \partial_x u_2)^2 dx dt = 0,
\end{equation}
and so appealing to \eqref{B_1-B_2} and \eqref{A_1-A_2} we have $A_1 = A_2$ and $B_1 = B_2$.
Then it is straightforward to show that $u_1=u_2$ by multiplying \eqref{hj differences} by $u = u_1-u_2$ and integrating by parts.
Noting that $u_x$ has an a priori bound, we have
\begin{equation}
\frac{1}{2}\int_0^L u^2(t,x)dx + \frac{\sigma^2}{4}\int_t^T\int_0^L u_x^2 dx dt
\leq C\int_t^T\int_0^L u^2 dx dt,
\end{equation}
and we conclude that $u = 0$ by Gronwall's Lemma.
\end{proof}

\bibliographystyle{siam}
\bibliography{C:/mybib/mybib}
\end{document}